\documentclass[reqno,centertags, 12pt, draft]{amsart}

\usepackage{amsmath,amsthm,amscd,amssymb}

\usepackage{latexsym}

\usepackage{mathrsfs}

\newcommand{\bbN}{{\mathbb{N}}}

\newcommand{\bbR}{{\mathbb{R}}}

\newcommand{\bbC}{{\mathbb{C}}}

\newcommand{\calZ}{{\mathcal Z}}

\newcommand{\calW}{{\mathcal W}}

\newcommand{\ip}[2]{ \left(#1,#2 \right)}

\newcommand{\se}[2]{ \left \{#1(#2) \right \}_{#2=1}^\infty}

\newcommand{\Jab}{{J \left( \se{a}{n},\se{b}{n} \right) }}

\newcommand{\mean}[1]{\left \langle #1 \right \rangle}

\newcommand{\mat}[4]{\left( \begin{array}{cc}
#1 & #2 \\
#3 & #4
\end{array} \right)}

\newcommand{\no}{\nonumber}

\newcommand{\ti}{\tilde  }

\newcommand{\beq}{\begin{equation}}

\newcommand{\eeq}{\end{equation}}

\newcommand{\ba}{\begin{align}}

\newcommand{\ea}{\end{align}}

\newcommand{\refer}[6]{ #1, {\it #2}, #3 {\bf #4} (#5), #6.}

\allowdisplaybreaks

\numberwithin{equation}{section}

\newtheorem{theorem}{Theorem}[section]

\newtheorem{proposition}[theorem]{Proposition}

\newtheorem{lemma}[theorem]{Lemma}

\theoremstyle{definition}

\newtheorem{definition}[theorem]{Definition}

\theoremstyle{remark}

\newtheorem*{remark*}{Remark}

\begin{document}

\title[Random Jacobi Matrices with Growing Parameters]{Spectral and Dynamical Properties of Certain Random Jacobi Matrices With
Growing Parameters}

\author[J.\ Breuer]{Jonathan Breuer}

\address{Jonathan Breuer, Mathematics 253-37, California Institute of Technology, Pasadena, CA 91125, USA. 
Email: jbreuer@caltech.edu}

\date{}

\begin{abstract}
In this paper, a family of random Jacobi matrices, with off-diagonal terms that exhibit power-law growth, is studied.
Since the growth of the
randomness
is slower than that of these terms, it is possible to use methods applied in the study of Schr\"odinger operators with random
decaying potentials. A particular result of the analysis is the existence of operators with arbitrarily fast transport
whose spectral measure is zero dimensional. The results are applied to the infinite Dumitriu-Edelman model \cite{dum-edel} 
and its spectral properties are analyzed.
\end{abstract}

\maketitle

\sloppy
\section{Introduction}

For a sequence of positive numbers, $\se{a}{n}$, and a sequence, $\se{b}{n}$, of real
numbers, let $\Jab$ denote the Jacobi matrix with off-diagonal elements given by $\se{a}{n}$ and diagonal
elements given by $\se{b}{n}$ on the diagonal. That is:
\beq \label{jacobi}
J \left( \{a(n)\}_{n=1}^\infty,\{b(n)\}_{n=1}^\infty \right)=\left( \begin{array}{ccccc}
b(1)    & a(1) & 0      & 0      & \ldots \\
a(1)    & b(2) & a(2)    & 0      & \ldots \\
0      & a(2) & b(3)    & a(3)    & \ddots \\
\vdots & \ddots   & \ddots & \ddots & \ddots \\
\end{array} \right).
\eeq

For $\eta_1 \in (0,1)$ and $\lambda_1>0$, let $J_{\lambda_1,\eta_1}$ be the Jacobi matrix whose parameters are
$a_{\lambda_1, \eta_1}(n)=\lambda_1 n^{\eta_1}$ and
$b_{\lambda_1,\eta_1}(n) \equiv 0$.
Since $\eta_1<1$, $J_{\eta_1}$ is a self-adjoint operator on $\ell^2(\bbN)$ \cite{berez}.
For any such operator we may define the spectral measure, $\mu$, as the unique
measure satisfying
\beq \no
\ip{\delta_1}{(J-z)^{-1}\delta_1}=
\int_{\bbR} \frac{d\mu(x)}{x-z}
\quad z \in \bbC \setminus \bbR
\eeq
where $\ip{\cdot}{\cdot}$ denotes the inner product in $\ell^2$.
It follows from the work of Janas and Naboko \cite{janas-naboko} that
$J_{\lambda_1,\eta_1}$ has absolutely continuous spectrum covering the whole real line.

This paper deals with random perturbations of $J_{\lambda_1,\eta_1}$ that are weak, in the sense that the variance of the
perturbing random parameters grows like $\sim n^{\eta_2}$ with $\eta_2<\eta_1$.

The case $a(n) \equiv 1,\ \eta_2<0$ (``$\eta_1=0$'') and no
perturbation off the diagonal, is the extensively studied family of
one-dimensional discrete Schr\"odinger operators with a random
decaying potential \cite{delyon, De84, DSS, efgp, ku, simon-decay}.
For these Schr\"odinger operators it has been established that when
$\eta_2<-\frac{1}{2}$, the absolutely continuous spectrum of the
Laplacian is a.s.\ preserved. When $\eta_2>-\frac{1}{2}$, however,
the disorder wins over and the spectrum is pure point with
eigenfunctions that decay at a super-polynomial rate. At the
critical point ($\eta_2=-\frac{1}{2}$) the spectral behavior
exhibits a sensitive dependence on the coupling constant: The
generalized eigenfunctions decay polynomially and the spectral
measure is pure point or singular continuous according to whether
these generalized eigenfunctions are in $\ell^2$ or not (for a
comprehensive treatment of discrete Schr\"odinger operators with
random decaying potentials, see Section 8 of \cite{efgp}).

From this perspective, the extension presented in this paper is in
allowing growth of the off-diagonal terms. Intuitively, these terms
are responsible for transport, and thus, their growth should have an
effect on the spectrum similar to that of the decay of the potential
(diagonal terms). Indeed, a particular case of our analysis is that
of $\eta_2=0$, namely, the diagonal terms are i.i.d.\ random
variables. We show that the critical point here is
$\eta_1=\frac{1}{2}$. Below this value, the spectrum is a.s.\ pure
point, whereas above it the spectral measure is one-dimensional.

More generally, let $\se{X_\omega}{n}$ be a sequence of i.i.d.\ random variables.
Let $\se{Y_\omega}{n}$ be another such sequence (the distributions of the $X$'s and the $Y$'s need not be the same).
Assume the following is satisfied: \\
(i) For all $n$
\beq \label{zero-mean}
\mean{X_\omega(n)}=\mean{Y_\omega(n)}=0
\eeq
(where $\mean{f(\omega)}\equiv \int_\Omega f(\omega) dp$
 and $(\Omega, \mathcal{F}, dp)$ is the underlying probability space.)\\
(ii) For any $k \in \bbN$
\beq \label{bounded-moments}
\mean{\vert Y_\omega(n) \vert^k}<\infty, \quad \mean{\vert X_\omega(n) \vert^k}<\infty.
\eeq
(iii) \beq \label{equal-moments}
\mean{(X_\omega(n))^2}=1, \quad \mean{(Y_\omega(n))^2}=\frac{1}{4}
\eeq
(iv)
The common distribution of $X_\omega(n)$ is absolutely continuous with respect to Lebesgue measure. \\

Given the quadruple $\Upsilon=(\eta_1,\eta_2,\lambda_1,\lambda_2)$ with $0<\eta_1<1$, $\eta_2<\eta_1$ and
$\lambda_1, \lambda_2>0$,
define the two random sequences
$$b_{\Upsilon,\omega}(n) \equiv  b_{\lambda_2, \eta_2; \omega}(n) \equiv \lambda_2 n^{\eta_2} X_\omega(n), \quad
\alpha_{\lambda_2, \eta_2 ;\omega}(n) \equiv \lambda_2 n^{\eta_2} Y_\omega(n).$$ Let
$$a_{\Upsilon,\omega}(n) \equiv a_{\lambda_1,\eta_1}(n)+\alpha_{\lambda_2, \eta_2 ;\omega}(n)$$
and define
\beq \label{upsilon-ensemble}
J_{\Upsilon,\omega}=J\left( \se{a_{\Upsilon,\omega}}{n}, \se{b_{\Upsilon, \omega}}{n}\right).
\eeq

The assumptions on the parameters defining $J_{\Upsilon,\omega}$, do
not exclude the possibility that some of the off-diagonal terms will
vanish. However, with probability one, this may happen only a finite
number of times, so that $J_{\Upsilon,\omega}$ has an infinite part
with strictly positive off-diagonal entries. In the following, when
we refer to $J_{\Upsilon,\omega}$, we refer to this part.

We shall prove
\begin{theorem} \label{growing-weights}
For the model above, let $\gamma=\eta_1-\eta_2$ and let
$\Lambda=\frac{1}{2}\left( \lambda_2/\lambda_1\right)^2$. The
following holds with probability one:
\begin{enumerate}
\item If $\gamma>\frac{1}{2}$ the spectrum of $J_{\Upsilon,\omega}$ is $\bbR$ and $\mu_{\Upsilon,\omega}$,
the spectral measure of $J_{\Upsilon,\omega}$, is
one-dimensional, meaning that it does not give weight to sets of Hausdorff dimension less than $1$.
\item In the case $\gamma=\frac{1}{2}$ the spectrum of $J_{\Upsilon,\omega}$ is $\bbR$
and we have the following two possibilities:
\begin{enumerate}
\item If $\Lambda>1-\eta_1$, then $\mu_{\Upsilon,\omega}$
is pure point with eigenfunctions decaying like
\beq \no
|\psi_\omega^E(n)|^2 \sim n^{-(\Lambda+\eta_1)}.
\eeq
\item If $\Lambda \leq 1-\eta_1$, then $\mu_{\Upsilon,\omega}$
is purely singular continuous with exact Hausdorff dimension equal to $1-\frac{\Lambda}{(1-\eta_1)}$.
\end{enumerate}
\item If $\gamma<\frac{1}{2}$ then the spectrum is pure point with eigenfunctions decaying like
\beq \no
|\psi_\omega^E(n)|^2 \sim e^{-\Lambda n^{1-2\gamma}}.
\eeq
In this case, if $\eta_1>2\eta_2$ then the spectrum fills $\bbR$.
\end{enumerate}
\end{theorem}
\begin{remark*}
We say that a measure, $\mu$, has exact Hausdorff dimension, $\varrho$, if it is supported on a set of Hausdorff dimension
$\varrho$ and does not give weight to sets of Hausdorff dimension less than $\varrho$.
For more information concerning the decomposition of general measures with respect to their
Hausdorff-dimensional properties, consult \cite{hausdorff} and references therein.
\end{remark*}
\begin{remark*}
In analogy to the Schr\"odinger case, one would expect to have absolutely continuous spectrum for $\gamma>1/2$.
Unfortunately, though we believe this is true, one-dimensional spectral measure is all we could get.
\end{remark*}
\begin{remark*}
The requirement that $\se{X_\omega}{n}$ and $\se{Y_\omega}{n}$ be
identically distributed sequences is not really necessary and is
made here only to simplify the discussion.
\end{remark*}
In resemblance of the Schr\"odinger case, the proof of this theorem follows by analyzing the
asymptotics of solutions to the formal eigenfunction equation ``$J \psi=E\psi$''. Namely, we shall analyze
the solutions to the difference equation:
\beq \label{ev1}
a(n)\psi(n+1)+b(n)\psi(n)+a(n-1)\psi(n-1)=E\psi(n) \quad n>1.
\eeq
By a theorem of Kiselev and Last \cite[Theorem 1.2]{kiselev-last-dynamics}, the results obtained have implications
for the quantum dynamics associated with $J$; namely the behavior of a given vector $\psi$ under the operation of
the one parameter unitary group $e^{-itJ}$. More precisely, let $\hat{X}$ be the position operator defined by
\beq \no
\left( \hat{X}\psi \right)(n)=n\psi(n).
\eeq

Theorem 1.2 of \cite{kiselev-last-dynamics}, which can be seen to hold in our setting, says that
Theorem \ref{growing-weights} and Proposition \ref{decaying-sol} below lead to
\begin{theorem} \label{dynamics}
Assume $\gamma=\frac{1}{2}$ and $\Lambda \leq 1-\eta_1$ (namely, case 2(b) of Theorem
\ref{growing-weights}). Then for any $\varepsilon>0$, $m \in \bbN$, $T>0$ and $\psi \in \ell^2$,
\beq \label{dynamical-spreading}
\frac{1}{T}\int_0^T \left \vert \ip{\hat{X}^m e^{-itJ_{\Upsilon,\omega}}\psi}{e^{-itJ_{\Upsilon,\omega}}\psi} \right \vert dt
\geq C(\omega,\psi,m,\varepsilon)T^{\left(m/(1-\eta_1) \right)-\varepsilon}
\eeq
with probability one.
\end{theorem}
Note that $\eta_1$ may be chosen arbitrarily close to $1$ while the spectral measure may have any dimension in $[0,1)$.
Thus, by tuning the parameters we obtain operators with any local spectral dimensions having arbitrarily fast transport.

As an application of our general analysis, we study the Gaussian $\beta$ ensembles arising naturally in the
context of Random Matrix Theory:
The eigenvalue distribution functions for the three classical Gaussian ensembles are given by
\begin{equation} \label{ev-distribution}
f_{\beta,N}(E_1,\cdots,E_N) = \frac{1}{G_{\beta N}}\exp
\left(-\frac{1}{2}\sum_{j=1}^N E_j^2\right
)\prod_{1\leq j<k\leq N} | E_j-E_k |^{\beta}\
\end{equation}
with $\beta=1,2$ and $4$ for the Gaussian Orthogonal Ensemble, Gaussian Unitary Ensemble and Gaussian Symplectic
Ensemble respectively.

A family of random matrix ensembles, indexed by $\beta$, having $f_{\beta,N}$ as their eigenvalue distribution function,
for \emph{any} positive value of $\beta$, was recently constructed by Dumitriu and Edelman \cite{dum-edel}.
The matrices in these ensembles are finite random Jacobi matrices with the distribution of the off diagonal terms
depending on $\beta$:
\begin{definition}
Fix $\beta>0$. The random family of Jacobi
matrices
$J_{\beta,\omega} \equiv J \left( \se{a_{\beta,\omega}}{n},\se{b_{\beta,\omega}}{n} \right)$ is
defined by:
\begin{enumerate}
\item The random variables $\se{a_{\beta,\omega}}{n},\se{b_{\beta,\omega}}{n}$ are all independent.
\item $b_{\beta, \omega}(n)$ are all standard Gaussian variables (that is, with zero mean and
variance$=1$), irrespective of $\beta$ and $n$.
\item The probability distribution function of $a_{\beta,\omega}(n)$ is given by
\beq \label{chi-dist}
P\{\omega \mid a_{\beta,\omega}(n)<C \}=\frac{2}{\Gamma
\left( \frac{\beta n}{2} \right)}
\int_{0}^C x^{\beta n-1}e^{-x^2}dx.
\eeq
\end{enumerate}
\end{definition}

In \cite{dum-edel}, Dumitriu and Edelman showed that the eigenvalue distribution function of the
finite matrix, obtained as the restriction of
$J_{\beta,\omega}$ to the $N \times N$ upper left corner, is $f_{\beta,N}$
for any $\beta>0$.

From property 3 above, it follows that
\beq \label{mean-of-an}
\mean{a_{\beta,\omega}(n)} \equiv \int_\Omega a_{\beta,\omega}(n) d\omega=
\frac{\Gamma \left( \frac{\beta n+1}{2}\right)}{\Gamma \left( \frac{\beta n}{2}\right)}
=\sqrt{\frac{\beta n}{2}}\left(1-\frac{1}{4\beta n} \right)+
\mathcal{O}\left( \frac{1}{n^{\frac{3}{2}}} \right),
\eeq

\beq \label{variance-of-an}
\mean{\left(a_{\beta,\omega}(n) -\mean{a_{\beta,\omega}(n)} \right)^2}= \frac{1}{4}+\mathcal{O}\left( \frac{1}{n} \right).
\eeq

Thus we see that the family $J_{\beta,\omega}$ corresponds to the case
$\eta_1=1/2$, $\eta_2=0$ of the general matrices introduced above. Technically, the following theorem is not
a corollary of Theorem \ref{growing-weights}, because of the
$\mathcal{O}\left( n^{-1/2} \right)$ term in \eqref{mean-of-an}
and the $\mathcal{O} \left( n^{-1} \right)$ term in \eqref{variance-of-an}.
The proof of Theorem \ref{growing-weights}, however, is robust with respect to such a change, and we have

\begin{theorem} \label{spectral-measure-GbetaE}
For any $\beta$, the spectrum of $J_{\beta,\omega}$ is $\bbR$ with probability one.

If $\beta<2$, then, with probability one, the spectral measure, $\mu_{\beta,\omega}$,
corresponding to $J_{\beta,\omega}$ and
$\delta_1$, is pure point with eigenfunctions decaying as
\beq \no
|\psi_\omega(n)|^2 \sim n^{-(\frac{1}{2}+\frac{1}{\beta})}.
\eeq

If $\beta \geq 2$, then with probability one, for any $\varepsilon>0$,
$\mu_{\beta,\omega}$ has exact dimension $1-\frac{2}{\beta}$ with probability one.
Furthermore, for $\beta \geq 2$, we have that, almost surely,
\beq \label{dynamical-spreading1}
\frac{1}{T}\int_0^T \left \vert \ip{\hat{X}^m e^{-itJ_{\beta,\omega}}\psi}{e^{-itJ_{\beta,\omega}}\psi} \right \vert
dt
\geq C(\omega,\psi,m,\varepsilon)T^{2m-\varepsilon}
\eeq
for any $\psi$, $\varepsilon>0$ and $m$.
\end{theorem}

This result, without the dynamical part, was announced in \cite{bfs}.
We note that the analogous Circular $\beta$ Ensembles can be realized as eigenvalues of truncated CMV matrices.
This was shown by Killip and
Nenciu \cite{killip-nenciu} and later used by Killip and Stoiciu
in their analysis of level statistics for ensembles of random CMV matrices \cite{KS06}.
The bulk spectral properties of the appropriate matrices were analyzed by Simon
\cite[Section 12.7]{OPUC}.

The proof of Theorem \ref{growing-weights} is given in the next section.
Since the proof of the spectral part of Theorem \ref{spectral-measure-GbetaE} is precisely the same, it is not given separately.
As noted earlier, the dynamical part of our analysis (Theorem \ref{dynamics} and the corresponding statement in Theorem
\ref{spectral-measure-GbetaE}) follows immediately from Theorem \ref{growing-weights} and Proposition \ref{decaying-sol}, by
Theorem 1.2 of \cite{kiselev-last-dynamics}.

The method we use is a variation on the one used by Kiselev-Last-Simon \cite[Section 8]{efgp} in
their analysis of the Schr\"odinger case described above. A notable difference is the fact that, due to the growth
of the $a(n)$, the effective energy parameter, $\frac{E}{a(n)}$, vanishes in the limit. This, in addition to requiring a
modification in the technique of proof (see Lemma \ref{theta-asymptotics} and Proposition \ref{theta-sums} below),
leads to the fact that the asymptotics
of the generalized eigenfunctions are
constant over $\bbR$. At the critical point ($(\eta_1-\eta_2)=\frac{1}{2}$), this implies uniformity of the local
Hausdorff dimensions of the spectral measure.

A modified Combes-Thomas estimate, for operators with unbounded off-diagonal terms,
enters our analysis in the identification of the spectrum of $J_{\Upsilon,\omega}$. Such an estimate may be of
independent interest and thus is presented in the Appendix.


\section{Proof of Theorem \ref{growing-weights}}

We begin with a simple lemma that shows that, in a certain sense, $J_{\Upsilon,\omega}$ is a random relatively decaying
perturbation of $J_{\lambda_1,\eta_1}$.
\begin{lemma} \label{almost-sure-decay}
For any $\varepsilon>0$ there exists, with probability one, a constant
$C=C(\omega,\varepsilon)$ for which
\beq \label{almost-sure-decay1}
\left| \frac{ \left|n^{\eta_2} X_\omega(n) \right|}{n^{\eta_1}} \right| \leq
\frac{C}{n^{\gamma-\varepsilon}}
\eeq
and
\beq \label{almost-sure-decay2}
\left| \frac{ \left| n^{\eta_2}Y_\omega(n) \right|}{n^{\eta_1}} \right| \leq
\frac{C}{n^{\gamma-\varepsilon}}
\eeq
where $\gamma=\eta_1-\eta_2$.
\end{lemma}
\begin{proof}
By \eqref{bounded-moments} and Chebyshev's inequality we have for any $k \in \bbN$
\beq \no
P_n\equiv\mathcal{P}\left\{\omega \mid \left| \frac{ \left| n^{\eta_2} X_\omega(n) \right|}
{n^{\eta_1}} \right| \geq
\frac{1}{n^{\gamma-\varepsilon}} \right\} \leq \frac{C(k)}{n^{2k \varepsilon}}.
\eeq
By choosing $2k >\varepsilon^{-1}$ we see that
\beq \no
\sum_{n=1}^\infty P_n <\infty.
\eeq
\eqref{almost-sure-decay1} follows now from Borel-Cantelli. The proof of \eqref{almost-sure-decay2} is the same.
\end{proof}

As stated in the Introduction, we follow the strategy of \cite{efgp}.
In particular, we will deduce the spectral properties
of $J_{\Upsilon,\omega}$ from the asymptotics of the solutions to the corresponding eigenfunction equation.

In order to fix notation, for a given Jacobi matrix $\Jab$ and a fixed $E \in \bbR$, denote by $\psi^E$ a solution to the
equation
\beq \label{ev}
a(n)\psi^E(n+1)+b(n)\psi^E(n)+a(n-1)\psi^E(n-1)=E\psi^E(n) \quad n>1.
\eeq
It is customary to extend this equation to $n=1$ by defining $a(0)=1$. Clearly, the space of
sequences $\{\psi^E(n)\}_{n=0}^\infty$ solving \eqref{ev} is a two-dimensional vector space and
any such sequence is completely determined by its values at $0$ and $1$. We let
$\psi^E_\phi(n)$ stand for the solution of \eqref{ev} satisfying
\beq \label{psi-phi}
\psi^E_\phi(0)=\sin(\phi) \qquad \psi^E_\phi(1)=\cos(\phi).
\eeq
We note that, formally, $\psi^E_0(n)$ satisfies
\beq \no
J \psi^E_0=E \psi^E_0.
\eeq

Let
\beq \no
S^E(n)=\mat{\frac{E-b(n)}{a(n)}}{-\frac{a(n-1)}{a(n)}}{1}{0}
\eeq
and
\beq \no
T^E(n)=S^E(n) \cdot S^E(n-1)\cdots  S^E(1).
\eeq
Then, for any $\phi$,
\beq \no
\left( \begin{array}{c}
\psi^E_\phi(n+1) \\
\psi^E_\phi(n)
\end{array}
 \right)=T^E(n) \left( \begin{array}{c}
\psi^E_\phi(1) \\
\psi^E_\phi(0)
\end{array}
 \right),
\eeq
and so
\beq \no
T^E(n)=\mat{\psi^E_0(n+1)}{\psi^E_{\frac{\pi}{2}}(n+1)}{\psi^E_0(n)}{\psi^E_{\frac{\pi}{2}}(n)}.
\eeq

We call the matrices $S^E(n)$ defined above \emph{one-step transfer matrices}, and for the
matrices $T^E(n)$, we use the name \emph{$n$-step transfer matrices}. Our main technical result is
\begin{theorem} \label{transfer-asymp}
Let $J_{\Upsilon,\omega}$ be the family of random Jacobi matrices described
in the Introduction. Then, for any $E \in \bbR$, the following holds with probability one:
\begin{enumerate}
\item If $\gamma>\frac{1}{2}$
\beq \label{supercritical-asymp}
\lim_{n \rightarrow \infty} \frac{\log \parallel T^E_\omega(n) \parallel^2}{\log(n)}=-\eta_1
\eeq
\item If $\gamma=\frac{1}{2}$
\beq \label{critical-asymp}
\lim_{n \rightarrow \infty} \frac{\log \parallel T^E_\omega(n) \parallel^2}{\log(n)}=
\Lambda-\eta_1.
\eeq
\item If $\gamma<\frac{1}{2}$
\beq \label{subcritical-asymp}
\lim_{n \rightarrow \infty} \frac{\log \parallel T^E_\omega(n) \parallel^2}{n^{1-2\gamma}}=
\Lambda.
\eeq
\end{enumerate}
\end{theorem}

The EFGP transform (see \cite{efgp}) is a useful tool for studying the asymptotic behavior of
$\parallel T^E(n) \parallel$ in the Schr\"odinger case ($a(n) \equiv 1$). For $a(n) \rightarrow \infty$, certain
modifications are needed. We proceed to present a version that is suitable for our purposes.

Let $J \left( \{a_\omega(n)\}_{n=1}^\infty \{b_\omega(n)\}_{n=1}^\infty \right)$ be a Jacobi matrix whose entries are
all independent random variables. Let $\ti{a}(n)=\mean{a_\omega(n)}$ and $\alpha_\omega=a_\omega(n)-\ti{a}(n)$, and
assume that
\beq \label{an-growth}
\lim_{n \rightarrow \infty}\ti{a}(n)=\infty
\eeq
and that
\beq \label{as-decay}
\lim_{n \rightarrow \infty}\frac{\alpha_\omega(n)}{\ti{a}(n)}=0
\eeq
with probability one.
These properties clearly hold for $J_{\Upsilon,\omega}$ (see Lemma \ref{almost-sure-decay}).
In the analysis that follows we keep $E \in \bbR$ fixed so we omit it from the notation.
Define
\beq \no
K_\omega(n)= \mat{1}{0}{0}{a_\omega(n)}.
\eeq
Then,
\beq \no
\ti{S}_\omega(n) \equiv K_\omega(n)S(n)K_\omega(n-1)^{-1}=
\mat{\frac{E-b_\omega(n)}{a_\omega(n)}}{-\frac{1}{a_\omega(n)}}{a_\omega(n)}{0}
\eeq
and
\beq \no
\ti{T}_\omega(n) \equiv \ti{S}_\omega(n) \cdot \ti{S}_\omega(n-1)\cdots\ti{S}_\omega(1) =
K_\omega(n) T_\omega(n).
\eeq
For any $\phi$, define the sequences $\se{u_{\omega,\phi}}{n}$ and
$\se{v_{\omega,\phi}}{n}$ by
\beq \no
\left( \begin{array}{c}
u_{\omega,\phi}(n) \\
v_{\omega,\phi}(n)
\end{array} \right)=\ti{T}_\omega(n) \left( \begin{array}{c}
\cos(\phi) \\
\sin(\phi)
\end{array} \right),
\eeq
so that, from the definition of $\psi_{\omega,\phi}$($=\psi_\phi$ for the random Jacobi
parameters), we see that
\beq \label{kn-solutions}
u_{\omega,\phi}(n)=\psi_{\omega,\phi}(n+1) \qquad
v_{\omega,\phi}(n)=a_\omega(n)\psi_{\omega,\phi}(n).
\eeq

By \eqref{an-growth} we see that for any $E \in \bbR$ and sufficiently
large $n$, we may define $k_n \in (0, \pi)$ by
\beq \label{kn}
2 \cos(k_n)=\frac{E}{\ti{a}(n)}.
\eeq
Clearly, $k_n \rightarrow \frac{\pi}{2}$ as $n \rightarrow \infty$.

Now, define $R_{\omega,\phi}(n)$ and $\theta_{\omega,\phi}(n)$ through
\beq \label{EFGP1}
R_{\omega,\phi}(n) \sin(\theta_{\omega,\phi}(n))=v_{\omega,\phi}(n) \sin(k_n)
\eeq
and
\beq \label{EFGP2}
R_{\omega,\phi}(n) \cos(\theta_{\omega,\phi}(n))=\ti{a}(n)u_{\omega,\phi}(n)-
v_{\omega,\phi}(n)\cos(k_n)
\eeq
so that (using \eqref{kn-solutions})
\beq \no
\begin{split}
R_{\omega,\phi}(n)^2&=v_{\omega,\phi}(n)^2+
\ti{a}(n)^2u_{\omega,\phi}(n)^2-2\ti{a}(n)u_{\omega,\phi}(n)v_{\omega,\phi}(n)\cos(k_n) \\
&= a_\omega(n)^2\psi_{\omega,\phi}(n)^2+
\ti{a}(n)^2 \psi_{\omega,\phi}(n+1)^2 \\
& \quad -a_\omega(n) E \psi_{\omega,\phi}(n+1)
\psi_{\omega,\phi}(n),
\end{split}
\eeq
which leads to
\beq \label{EFGP-compare}
\begin{split}
\frac{R_{\omega,\phi}(n)^2}
{\ti{a}(n)^2\left( \psi_{\omega,\phi}(n)^2+\psi_{\omega,\phi}(n+1)^2 \right)}
&=1+\frac{2 \alpha_\omega(n) \ti{a}(n) \psi_{\omega,\phi}(n)^2}
{\ti{a}(n)^2\left( \psi_{\omega,\phi}(n)^2+\psi_{\omega,\phi}(n+1)^2 \right)} \\
& \quad +\frac{2 \alpha_\omega(n)^2 \psi_{\omega,\phi}(n)^2}
{\ti{a}(n)^2\left( \psi_{\omega,\phi}(n)^2+\psi_{\omega,\phi}(n+1)^2 \right)} \\
&\quad -\frac{Ea_\omega(n)\psi_{\omega,\phi}(n) \psi_{\omega,\phi}(n+1)}
{\ti{a}(n)^2\left( \psi_{\omega,\phi}(n)^2+\psi_{\omega,\phi}(n+1)^2 \right)}.
\end{split}
\eeq
By \eqref{as-decay}, it follows that the right hand side converges to one with
probability $1$, uniformly in $\phi$, so that almost surely, for sufficiently large $n$, there
are constants $C_1,C_2>0$ such that
\beq \no
C_1R_{\omega,\phi}(n)^2
\leq\ti{a}(n)^2\left( \psi_{\omega,\phi}(n)^2+\psi_{\omega,\phi}(n+1)^2 \right)
\leq C_2R_{\omega,\phi}(n)^2
\eeq
Now, by a straightforward adaptation of Lemma 2.2 of \cite{efgp} it follows that for any two
angles $\phi_1 \neq \phi_2$, there are constants $C_3,C_4>0$ such that
\beq \label{compare-Rn-transfer}
\begin{split}
C_3 \max(R_{\omega,\phi_1}(n)^2 ,R_{\omega,\phi_2}(n)^2)
&\leq \ti{a}(n)^2 \parallel T_\omega(n) \parallel^2 \\
&\leq C_4 \max(R_{\omega,\phi_1}(n)^2 ,R_{\omega,\phi_2}(n)^2).
\end{split}
\eeq
Thus, we are led to examine the asymptotic properties of $\log R_{\omega,\phi}(n)$.

Let us formulate a recursion relation for $R_{\omega,\phi}(n)^2$:
\eqref{EFGP1} and \eqref{EFGP2} mean
\beq \no
\begin{split}
&\left( \begin{array}{c}
R_{\omega,\phi}(n) \sin(\theta_{\omega,\phi}(n)) \\
R_{\omega,\phi}(n) \cos(\theta_{\omega,\phi}(n))
\end{array} \right) \\
& = \mat{0}{\sin(k_{n})}{\ti{a}(n)}{-\cos(k_{n})}
\left( \begin{array}{c}
u_{\omega,\phi}(n) \\
v_{\omega,\phi}(n)
\end{array} \right) \\
& = \mat{0}{\sin(k_{n})}{\ti{a}(n)}{-\cos(k_{n})}
\mat{1}{0}{0}{a_\omega(n)}
\left( \begin{array}{c}
\psi_{\omega,\phi}(n+1) \\
\psi_{\omega,\phi}(n)
\end{array} \right).
\end{split}
\eeq
We also know
\beq \no
\left( \begin{array}{c}
\psi_{\omega,\phi}(n+2) \\
\psi_{\omega,\phi}(n+1)
\end{array} \right)
=S_\omega(n+1)
\left( \begin{array}{c}
\psi_{\omega,\phi}(n+1) \\
\psi_{\omega,\phi}(n)
\end{array} \right),
\eeq
so
\beq \label{Rn-recursion1}
\begin{split}
&\left( \begin{array}{c}
R_{\omega,\phi}(n+1) \sin(\theta_{\omega,\phi}(n+1)) \\
R_{\omega,\phi}(n+1) \cos(\theta_{\omega,\phi}(n+1))
\end{array} \right) \\
& = \mat{0}{\sin(k_{n+1})}{\ti{a}(n+1)}{-\cos(k_{n+1})}
\mat{1}{0}{0}{a_\omega(n+1)}
S_\omega(n+1) \\
& \quad \cdot
\mat{1}{0}{0}{a_\omega(n)}^{-1}
\mat{0}{\sin(k_{n})}{\ti{a}(n)}{-\cos(k_{n})}^{-1}
\left( \begin{array}{c}
R_{\omega,\phi}(n) \sin(\theta_{\omega,\phi}(n)) \\
R_{\omega,\phi}(n) \cos(\theta_{\omega,\phi}(n))
\end{array} \right) \\
& = \mat{0}{\sin(k_{n+1})}{\ti{a}(n+1)}{-\cos(k_{n+1})}
\ti{S}_\omega(n+1) \\
& \quad \cdot\mat{0}{\sin(k_{n})}{\ti{a}(n)}{-\cos(k_{n})}^{-1}
\left( \begin{array}{c}
R_{\omega,\phi}(n) \sin(\theta_{\omega,\phi}(n)) \\
R_{\omega,\phi}(n) \cos(\theta_{\omega,\phi}(n))
\end{array} \right). \\
\end{split}
\eeq
Now, write
\beq \no
\begin{split}
\ti{S}_\omega(n+1)&=\mat{\frac{E-b_\omega(n+1)}{a_\omega(n+1)}}{-\frac{1}{a_\omega(n+1)}}
{a_\omega(n+1)}{0} \\
&=\frac{\ti{a}(n+1)}{a_\omega(n+1)}
\Bigg(
\mat{\frac{E}{\ti{a}(n+1)}}{-\frac{1}{\ti{a}(n+1)}}{\ti{a}(n+1)}{0} \\
& \quad +
\mat{-\frac{b_\omega(n+1)}{\ti{a}(n+1)}}{0}{\frac{a_\omega(n+1)^2-\ti{a}(n+1)^2}{\ti{a}(n+1)}}{0}
\Bigg).
\end{split}
\eeq

We define
\beq \no
\begin{split}
\calZ_\omega(n+1)&=\mat{0}{\sin(k_{n+1})}{\ti{a}(n+1)}{-\cos(k_{n+1})}
\mat{\frac{E}{\ti{a}(n+1)}}{-\frac{1}{\ti{a}(n+1)}}{\ti{a}(n+1)}{0} \\
& \quad \cdot
\mat{0}{\sin(k_{n})}{\ti{a}(n)}{-\cos(k_{n})}^{-1}
\left( \begin{array}{c}
\sin(\theta_{\omega,\phi}(n)) \\
\cos(\theta_{\omega,\phi}(n))
\end{array} \right),
\end{split}
\eeq
and
\beq \no
\begin{split}
\calW_\omega(n+1)&=\mat{0}{\sin(k_{n+1})}{\ti{a}(n+1)}{-\cos(k_{n+1})}
\mat{-\frac{b_\omega(n+1)}{\ti{a}(n+1)}}{0}{\frac{a_\omega(n+1)^2-\ti{a}(n+1)^2}{\ti{a}(n+1)}}{0} \\
& \quad \cdot
\mat{0}{\sin(k_{n})}{\ti{a}(n)}{-\cos(k_{n})}^{-1}
\left( \begin{array}{c}
\sin(\theta_{\omega,\phi}(n)) \\
\cos(\theta_{\omega,\phi}(n))
\end{array} \right)
\end{split}
\eeq
(we ignore the dependence on $\phi$ since we keep it fixed).
Then, from \eqref{Rn-recursion1} we see that
\beq \label{Rn-recursion2}
\frac{R_{\omega,\phi}(n+1)^2}{R_{\omega,\phi}(n)^2}
=\frac{\ti{a}(n+1)^2}{a_\omega(n+1)^2}
\parallel \calZ_\omega(n+1)+\calW_\omega(n+1) \parallel^2.
\eeq

$\theta_{\omega,\phi}(n)$ satisfies a recurrence relation as well:
From \eqref{kn-solutions} we have that
\beq \label{v-u}
v_{\omega,\phi}(n+1)=a_\omega(n+1)u_{\omega,\phi}(n)
\eeq
and
\beq \label{schrodinger-u}
a_\omega(n+1)u_{\omega,\phi}(n+1)+b_\omega(n+1)u_{\omega,\phi}(n)+
a_{\omega}(n)u_{\omega,\phi}(n-1)=Eu_{\omega,\phi}(n).
\eeq
Write, using \eqref{EFGP1}-\eqref{EFGP2}
\beq \label{cot1}
\begin{split}
\cot(\theta_{\omega,\phi}(n+1))&=\frac{\ti{a}(n+1)u_{\omega,\phi}(n+1)-
\cos(k_{n+1})v_{\omega,\phi}(n+1)}{\sin(k_{n+1})v_{\omega,\phi}(n+1)} \\
&=\frac{\ti{a}(n+1)u_{\omega,\phi}(n+1)-
a_\omega(n+1)\cos(k_{n+1})u_{\omega,\phi}(n)}{\sin(k_{n+1})a_\omega(n+1)u_{\omega,\phi}(n)}.
\end{split}
\eeq
Furthermore, observing that
\beq \no
R_{\omega,\phi}(n)\sin(\theta_{\omega,\phi}(n)+k_n)= \ti{a}(n)\sin(k_n)u_{\omega,\phi}(n),
\eeq
and
\beq \no
\begin{split}
R_{\omega,\phi}(n)\cos(\theta_{\omega,\phi}(n)+k_n)&= \ti{a}(n)\cos(k_n)u_{\omega,\phi}(n)-
v_{\omega,\phi}(n) \\
&= \ti{a}(n)\cos(k_n)u_{\omega,\phi}(n)-
a_\omega(n)u_{\omega,\phi}(n-1),
\end{split}
\eeq
we may write
\beq \label{cot2}
\cot(\theta_{\omega,\phi}(n)+k_n)=\frac{\ti{a}(n)\cos(k_n)u_{\omega,\phi}(n)-
a_\omega(n)u_{\omega,\phi}(n-1)}{\ti{a}(n)\sin(k_n)u_{\omega,\phi}(n)}.
\eeq

Substituting $u_{\omega,\phi}(n+1)$ from \eqref{schrodinger-u} into \eqref{cot1} and then
$u_{\omega,\phi}(n-1)$ from \eqref{cot2} into the resulting equation, we get
\beq \label{recursion-for-theta}
\begin{split}
\cot(\theta_{\omega,\phi}(n+1))&=\frac{\ti{a}(n+1)\ti{a}(n)}{a_\omega(n+1)^2}
\frac{\sin(k_{n})}{\sin(k_{n+1})}\cot(\theta_{\omega,\phi}(n)+k_n) \\
&\quad+\cot(k_{n+1})\left(\frac{\ti{a}(n+1)^2}{a_\omega(n+1)^2}-1 \right) \\
&-\frac{\ti{a}(n+1)}{\sin(k_n)a_\omega(n+1)^2}b_\omega(n+1) \\
&\equiv \kappa_\omega(n+1)\cot(\bar{\theta}_{\omega,\phi}(n))+\zeta_\omega(n+1)
\end{split}
\eeq
where
\beq \no
\bar{\theta}_{\omega,\phi}(n)=\theta_{\omega,\phi}(n)+k_n.
\eeq

By \eqref{compare-Rn-transfer} and picking two angles $\phi_1 \neq \phi_2$, Theorem \ref{transfer-asymp} follows from
\begin{proposition} \label{transfer-asymp1}
Let $J_{\Upsilon,\omega}$ be the family of random Jacobi matrices described
in the Introduction. Then, for any $E \in \bbR$, and for any $\phi$, the following holds with probability one:
\begin{enumerate}
\item If $\gamma>\frac{1}{2}$
\beq \label{supercritical-asymp1}
\lim_{n \rightarrow \infty} \frac{\log R_{\omega,\phi}^E(n)^2}{\log(n)}=\eta_1
\eeq
\item If $\gamma=\frac{1}{2}$
\beq \label{critical-asymp1}
\lim_{n \rightarrow \infty} \frac{\log R_{\omega,\phi}^E(n)^2}{\log(n)}=
\Lambda+\eta_1.
\eeq
\item If $\gamma<\frac{1}{2}$
\beq \label{subcritical-asymp1}
\lim_{n \rightarrow \infty} \frac{\log R^E_{\omega,\phi}(n)^2}{n^{1-2\gamma}}=
\Lambda.
\eeq
\end{enumerate}
\end{proposition}
\begin{proof}
As in \cite{efgp} we shall prove the statement by using the recursion relation for $R_\omega(n)^2$
(equation \eqref{Rn-recursion2}). Namely, we shall
prove that
\beq \label{recursion-asymp}
\frac{1}{F_\gamma(n)}
\sum_{j=1}^n \left(
\log \left(\parallel \calZ_\omega(j)+\calW_\omega(j) \parallel^2 \right)-
\log \left( \frac{a_\omega(j)^2}{\ti{a}(j)^2} \right) \right)
\eeq
converges to the appropriate limit, where $F_\gamma(n)=\log(n)$ for
$\gamma \geq \frac{1}{2}$ and $F_\gamma(n)=n^{1-2\gamma}$ otherwise.
From this point on, $a_\omega(n)=a_{\Upsilon,\omega}(n)$ and $b_\omega(n)=b_{\Upsilon,\omega}(n)$.

We shall need some estimate on the behavior of $\theta_\omega(n)$. We start with

\begin{lemma} \label{theta-asymptotics}
For any $\varepsilon>0$, there exists, with probability one, a constant
$\ti{C}=\ti{C}(\omega,\varepsilon)$, such that
\beq \label{theta-decay}
|\theta_{\omega,\phi}(n+1)-\bar{\theta}_{\omega,\phi}(n)| \leq
\ti{C} \max (n^{-\gamma+\varepsilon}, n^{-1}).
\eeq
\end{lemma}
\begin{proof}[Proof of the Lemma]
We start by proving a similar statement for $|\kappa_\omega(n+1)-1|+|\zeta_\omega(n+1)|$:

Recall that $k_n \rightarrow \frac{\pi}{2}$. Thus, for large enough $n$,
$\sin(k_n)>\frac{1}{2}$. Moreover, $\cos(k_n) \sim \frac{1}{n^{\eta_1}}$. It follows that
\beq \no
\begin{split}
|\sin(k_n)-\sin(k_{n+1})|=\frac{|\cos(k_n)^2-\cos(k_{n+1})^2|}{|\sin(k_n)+\sin(k_{n+1})|}
\leq \frac{C'}{n^{1+2\eta_1}}.
\end{split}
\eeq
Thus,
\beq \no
\begin{split}
\left| \kappa_\omega(n+1)-1 \right|+\left| \zeta_\omega(n+1) \right|&\leq
\left| \frac{\sin(k_{n})}{\sin(k_{n+1})} \left( \frac{\ti{a}(n+1)\ti{a}(n)}{a_\omega(n+1)^2}
 -1\right) \right| \\
&\quad+\left| \frac{\sin(k_n)-\sin(k_{n+1})}{\sin(k_{n+1})} \right| \\
&\quad+\left| \cot(k_{n+1})\left(\frac{\ti{a}(n+1)^2}{a_\omega(n+1)^2}-1 \right) \right| \\
&\quad+\left| \frac{\ti{a}(n+1)}{\sin(k_n)a_\omega(n+1)^2}b_\omega(n+1) \right| \\
&=I_1+I_2+I_3+I_4.
\end{split}
\eeq
Now, choosing $\varepsilon<\eta_1-\eta_2$, by Lemma \ref{almost-sure-decay}
\beq \no
\begin{split}
I_1 &\leq 2 \left|\frac{\ti{a}(n+1)\left(\ti{a}(n)-\ti{a}(n+1)\right)}{a_\omega(n+1)^2} \right|
+4\left|\frac{\alpha_\omega(n+1) \ti{a}(n+1)}{a_\omega(n+1)^2} \right| \\
&\quad+2\left|\frac{\alpha_\omega(n+1)^2}{a_\omega(n+1)^2} \right|
\leq C_1(\omega,\varepsilon)\max(n^{-\gamma+\varepsilon}, n^{-1})
\end{split}
\eeq
\beq \no
I_3 \leq 2|\cos(k_{n+1})| \left|\frac{2 \alpha_\omega(n+1)\ti{a}(n+1)+\alpha_\omega(n+1)^2}
{a_\omega(n+1)^2} \right|\leq \frac{C_3(\omega,\varepsilon)}{n^{\gamma-\varepsilon}},
\eeq
and
\beq \no
I_4 \leq 2\left| \frac{\ti{a}(n+1)b_\omega(n+1)}{a_\omega(n+1)^2} \right|
\leq \frac{C_4(\omega,\varepsilon)}{n^{\gamma-\varepsilon}}
\eeq
almost surely.
We have shown above that
\beq \no
I_2 \leq \frac{C_2}{n^{1+\eta_1}}
\eeq
so we see that there exists, with probability one, a constant
$C_5(\omega,\varepsilon)$ for which
\beq \label{C5}
|\kappa_\omega(n+1)-1|+|\zeta_\omega(n+1)|\leq \frac{C_5} \max (n^{-\gamma+\varepsilon}, n^{-1}).
\eeq

Now, use
\beq \no
e^{2i x}=1-\frac{2}{1+i\cot x}
\eeq
and \eqref{recursion-for-theta} to see that, if $|\kappa_\omega(n+1)-1|+2|\zeta_\omega(n+1)|<\frac{1}{2}$,
(which indeed happens almost surely, for large enough $n$), then
\beq \no
\left| e^{2i\theta_{\omega,\phi}(n+1)}-e^{2i\bar{\theta}_{\omega,\phi}(n)} \right| \leq 4
\left(
\left|\kappa_\omega(n+1)-1 \right|+\left|\zeta_\omega(n+1) \right| \right),
\eeq
which implies (by $|e^{ix}-1| \geq \frac{2|x|}{\pi}$) that
\beq \no
|\theta_{\omega,\phi}(n+1)-\bar{\theta}_{\omega,\phi}(n)| \leq \pi \left(
\left|\kappa_\omega(n+1)-1 \right|+\left|\zeta_\omega(n+1) \right| \right).
\eeq
This, together with \eqref{C5}, implies \eqref{theta-decay} and concludes the proof of the
lemma.
\end{proof}

The direct consequence of this is
\begin{proposition} \label{theta-sums}
Assume $f(n)$ is a function that satisfies
\beq \no
f(n)=o \left( n^r \right)
\eeq
and
\beq \no
f(n+1)-f(n)=o \left( n^{r-1} \right),
\eeq
with
\beq \label{r-value}
r=\left \{ \begin{array}{cc}
\gamma-1 & {\rm if } \gamma \geq 1/2 \  {\rm and } \ \eta_2>0 \\
-\gamma  & {\rm if } \gamma < 1/2  \ {\rm and } \ \eta_2>0 \\
\eta_1-1 & {\rm if } \gamma \geq 1/2 \  {\rm and } \ \eta_2 \leq 0 \\
\eta_1-2\gamma & {\rm if } \gamma< 1/2 \  {\rm and } \ \eta_2 \leq 0.
\end{array} \right.
\eeq
Then, for any $\phi$,
\beq \label{theta-sum}
\lim_{n \rightarrow \infty}\frac{1}{F_\gamma(n)} \sum_{j=1}^n
f(j) \cos(2 \theta_{\omega,\phi}(j))=0
\eeq
almost surely.
The same statement holds with $\theta$ replaced by $\bar{\theta}$.
\end{proposition}
\begin{proof}[Proof of the Proposition]
We shall prove the statement for $\theta$. By summation by parts,
\beq \label{abel}
\begin{split}
& \left| \sum_{j=1}^n f(j)\cos(2\theta_{\omega,\phi}(j)) \right| \\
& \quad =\left| f(n)\sum_{j=1}^n \cos(2\theta_{\omega,\phi}(j))-
\sum_{j=1}^{n-1}\sum_{l=1}^j \cos(2\theta_{\omega,\phi}(l)) \left( f(j+1)-f(j) \right) \right|. \\
\end{split}
\eeq
Thus we are led to examine $\sum_{l=1}^j \cos(2\theta_{\omega,\phi}(l))$. Assume that $j=2m$ is
even. Then
\beq \label{even}
\begin{split}
\left| \sum_{l=1}^j \cos(2\theta_{\omega,\phi}(l)) \right|&=
\left| \sum_{l=1}^m \left( \cos(2\theta_{\omega,\phi}(2l))
-\cos(2\theta_{\omega,\phi}(2l-1)+\pi)\right) \right| \\
&\leq \sum_{l=1}^m \left| 2\theta_{\omega,\phi}(2l)
-2\theta_{\omega,\phi}(2l-1)-\pi \right| \\
&\leq 2 \sum_{l=1}^m \left| \theta_{\omega,\phi}(2l)
-\bar{\theta}_{\omega,\phi}(2l-1) \right|+\left| k_{2l-1}-\frac{\pi}{2} \right|\\
&\leq C_\omega \max( j^{1-\gamma+\varepsilon}, j^{1-\eta_1})
\end{split}
\eeq
almost surely, by Lemma \ref{theta-asymptotics} and by
\beq \no
\left| k_n -\frac{\pi}{2} \right| \leq 2|\cos(k_n)|,
\eeq
which holds for sufficiently large $n$.
Thus, for any $j$, we get
\beq \label{odd}
\begin{split}
\left| \sum_{l=1}^j \cos(2\theta_{\omega,\phi}(l)) \right| & \leq
C_\omega \max( j^{1-\gamma+\varepsilon}, j^{1-\eta_1})
+1 \\
& \leq (C_\omega+1) \max ( j^{1-\gamma+\varepsilon}, j^{1-\eta_1}).
\end{split}
\eeq
A simple calculation finishes the proof for $\theta$.
The proof for $\bar{\theta}$ follows the same argument, with an additional $|k_{2l}-k_{2l-1}|$
term in \eqref{even}.
\end{proof}

We abbreviate
\beq \label{Aj}
\begin{split}
A_\omega(j) &= \frac{a_\omega(j)^2-\ti{a}(j)^2}{\ti{a}(j)^2}=
2\frac{\alpha_\omega(j)}{\ti{a}(j)}+\frac{\alpha_\omega(j)^2}{\ti{a}(j)^2} \\
& =2\frac{\lambda_2}{\lambda_1}\frac{Y_\omega(j)}{j^{\gamma}}
+\left( \frac{\lambda_2}{\lambda_1} \right)^2\frac{Y_\omega(j)^2}{j^{2\gamma}},
\end{split}
\eeq
and
\beq \label{Bj}
B_\omega(j)=\frac{b_\omega(j)}{\ti{a}(j)}=\frac{\lambda_2}{\lambda_1}\frac{X_\omega(j)}{j^\gamma}
\eeq
By a straightforward calculation we get
\beq \label{W}
\begin{split}
\parallel \calW_\omega(j) \parallel^2&=
\frac{\sin^2(\bar{\theta}_\omega(j-1))\ti{a}(j)^2}{\sin^2(k_{j-1}) \ti{a}(j-1)^2} \\
&  \quad \times
\Bigg(A_\omega(j)^2 +B_\omega(j)^2+
2\cos(k_j)B_\omega(j)A_\omega(j)  \Bigg),
\end{split}
\eeq
\beq \label{Z}
\begin{split}
\parallel \calZ_\omega(j) \parallel^2 &= 1+\sin^2(\bar{\theta}_\omega(j-1))
\frac{\ti{a}(j)^2-\ti{a}(j-1)^2}{\ti{a}(j-1)^2} \\
& \quad +\frac{\ti{a}(j)^2 \sin^2(\bar{\theta}_\omega(j-1))}{\ti{a}(j-1)^2}
\frac{\sin^2(k_j)-\sin^2(k_{j-1})}{\sin^2(k_{j-1})}
\end{split}
\eeq
and
\beq \label{ZW}
\begin{split}
\left( \calZ_\omega(j),\calW_\omega(j) \right)&=
\frac{\ti{a}(j)}{\ti{a}(j-1)} \\
& \quad \times \Bigg(\frac{A_\omega(j) \sin(\bar{\theta}_\omega(j-1))}{\sin^2(k_{j-1})}
\Big(\frac{\ti{a}(j)}{\ti{a}(j-1)}\sin(\bar{\theta}_\omega(j-1)) \\
&-\cos(k_j)\sin(\theta_\omega(j-1))\Big)
-\frac{B_\omega(j)\sin(2\bar{\theta}_\omega(j-1))}{2\sin(k_{j-1})}
\Bigg)
\end{split}
\eeq
(Recall
$\bar{\theta}_{\omega}(j) \equiv \bar{\theta}_{\omega,\phi}(j)=\theta_{\omega,\phi}(j)+k_j
\equiv \theta_{\omega}(j)+k_j$.)

Let $\varepsilon<\frac{\gamma}{8}$.
By Theorem \ref{almost-sure-decay}, the fact that
$\cos(k_j) \sim \frac{1}{\ti{a}(j)}$ and the identity
\beq \no
|\sin^2(x)-\sin^2(y)|=|\cos^2(x)-\cos^2(y)|,
\eeq
it follows that
$\parallel \calW_\omega(j) \parallel^2=\mathcal{O}_\omega(j^{-2\gamma+2\varepsilon})$,
$\left(\calZ_\omega(j), \calW_\omega(j)\right)=\mathcal{O}_\omega(j^{-\gamma+\varepsilon})$ and
$\left( \parallel \calZ_\omega(j) \parallel^2-1 \right)=\mathcal{O}(j^{-1})$
with probability $1$, where the notation $\mathcal{O}_\omega$ indicates that the implicit constant depends on $\omega$.
Thus, we can use
\beq \no
\log(1+x)=x-\frac{x^2}{2}+\mathcal{O}(x^3),
\eeq
together with the observation that (almost surely)
\beq \no
\begin{split}
&\Big(\parallel \calZ_\omega(j) \parallel^2-1+2(\calZ_\omega(j),\calW_\omega(j))
+\parallel \calW_\omega(j) \parallel^2 \Big)^2 \\
& \quad = 4(\calZ_\omega(j),\calW_\omega(j))^2
+\mathcal{O}_\omega\left( \frac{1}{j^{3\gamma-\varepsilon}} \right),
\end{split}
\eeq
to see that, with probability one, for large enough $j$ we have that
\beq \label{principal}
\begin{split}
& \log \Big(1+ \parallel \calZ_\omega(j) \parallel^2-1+2(\calZ_\omega(j),\calW_\omega(j))
+\parallel \calW_\omega(j) \parallel^2 \Big) \\
&\quad=(\parallel \calZ_\omega(j) \parallel^2-1)+2(\calZ_\omega(j),\calW_\omega(j))
+\parallel \calW_\omega(j) \parallel^2 \\
&\qquad-2(\calZ_\omega(j),\calW_\omega(j))^2+\mathcal{O}_\omega
\left( \frac{1}{j^{3\gamma-\varepsilon}}+\frac{1}{j^{1+\gamma-\varepsilon}} \right).
\end{split}
\eeq
Therefore, since the edition of a finite number of terms is inconsequential, it follows that
\beq \label{sum-for-log1}
\begin{split}
&\lim_{n \rightarrow \infty} \frac{1}{F_\gamma(n)}
\sum_{j=1}^n
\log \left(\parallel \calZ_\omega(j)+\calW_\omega(j) \parallel^2 \right) \\
&\quad=
\lim_{n \rightarrow \infty} \frac{1}{F_\gamma(n)}
\sum_{j=1}^n
\Big( (\parallel \calZ_\omega(j) \parallel^2-1)+2(\calZ_\omega(j),\calW_\omega(j))
+\parallel \calW_\omega(j) \parallel^2 \\
&\qquad-2(\calZ_\omega(j),\calW_\omega(j))^2 \Big)
\end{split}
\eeq
with probability one, in the sense that both limits exist together and are equal if they do.

Similarly,
\beq \label{sum-for-log2}
\begin{split}
\lim_{n \rightarrow \infty}\frac{-1}{F_\gamma(n)}
\sum_{j=1}^n
\log \left( \frac{a_\omega(j)^2}{\ti{a}(j)^2} \right)&=
\lim_{n \rightarrow \infty}\frac{-1}{F_\gamma(n)}
\sum_{j=1}^n\log \left(1+A_\omega(j)\right) \\
&=
\lim_{n \rightarrow \infty}\frac{-1}{F_\gamma(n)}
\sum_{j=1}^n \left( A_\omega(j)-\frac{A_\omega(j)^2}{2} \right).
\end{split}
\eeq

Thus, our problem is reduced to computing the limits:
\beq \no
\xi_{\calZ} \equiv \lim_{n \rightarrow \infty} \frac{1}{F_\gamma(n)}\sum_{j=1}^n
\left( \parallel \calZ_\omega(j) \parallel^2-1 \right)
\eeq
\beq \no
\xi_{\calW} \equiv \lim_{n \rightarrow \infty} \frac{1}{F_\gamma(n)}\sum_{j=1}^n
\parallel \calW_\omega(j) \parallel^2
\eeq
\beq \no
\xi_{\calZ \calW} \equiv \lim_{n \rightarrow \infty} \frac{2}{F_\gamma(n)}\sum_{j=1}^n
\left(\calZ_\omega(j),\calW_\omega(j) \right)
\eeq
\beq \no
\xi_{\calZ \calW_2} \equiv \lim_{n \rightarrow \infty}  \frac{-2}{F_\gamma(n)}\sum_{j=1}^n
\left(\calZ_\omega(j),\calW_\omega(j) \right)^2
\eeq
\beq \no
\xi_{A} \equiv \lim_{n \rightarrow \infty} \frac{-1}{F_\gamma(n)}\sum_{j=1}^n
\left(A_\omega(j)-\frac{A_\omega(j)^2}{2} \right).
\eeq
By \eqref{Z},
\beq \no
\begin{split}
\parallel \calZ_\omega(j) \parallel^2-1 &=
\sin^2(\bar{\theta}_\omega(j-1))
\frac{\ti{a}(j)^2-\ti{a}(j-1)^2}{\ti{a}(j-1)^2} \\
& \quad +\frac{\ti{a}(j)^2
\sin^2(\bar{\theta}_\omega(j-1))}{\ti{a}(j-1)^2}
\frac{\sin^2(k_j)-\sin^2(k_{j-1})}{\sin^2(k_{j-1})}.
\end{split}
\eeq
The last term on the right is absolutely summable ($=\mathcal{O}(n^{-1-2\eta_1})$) so we only need to look at
$\sin^2(\bar{\theta}_\omega(j-1)) \frac{\ti{a}(j)^2-\ti{a}(j-1)^2}{\ti{a}(j-1)^2}$. But, by
$\sin^2(\alpha)=\frac{1}{2}-\frac{\cos(2\alpha)}{2}$ and by Proposition \ref{theta-sums}, we see that (recall
$\ti{a}(j)=\lambda_1j^{-\eta_1}$)
\beq \label{xiZ}
\begin{split}
\xi_{\calZ}&=\frac{1}{2}
\lim_{n \rightarrow \infty} \frac{1}{F_\gamma(n)}\sum_{j=1}^n\frac{\ti{a}(j)^2-\ti{a}(j-1)^2}{\ti{a}(j-1)^2} \\
&=
\frac{1}{2}\lim_{n \rightarrow \infty} \frac{1}{F_\gamma(n)}\sum_{j=1}^n\frac{2\eta_1}{j}=
\left \{ \begin{array}{cc}
\eta_1 & {\rm if} \ \gamma \geq 1/2 \\
0 & {\rm otherwise.} \end{array} \right.
\end{split}
\eeq
For the other four limits, we shall use extensively Lemma 8.4 of \cite{efgp}, in order to replace $A_\omega$ and $B_\omega$ by their
means. For example, write
\beq \label{W1}
\begin{split}
\parallel \calW_\omega(j) \parallel^2 &=
\frac{\sin^2(\bar{\theta}_\omega(j-1))\ti{a}(j)^2}{\sin^2(k_{j-1})
\ti{a}(j-1)^2}
\Bigg(A_\omega(j)^2-\mean{A_\omega(j)^2} \Bigg) \\
& \quad+ \frac{\sin^2(\bar{\theta}_\omega(j-1))\ti{a}(j)^2}{\sin^2(k_{j-1})
\ti{a}(j-1)^2}
\Bigg(B_\omega(j)^2-\mean{B_\omega(j)^2}\Bigg) \\
&\quad+\frac{\sin^2(\bar{\theta}_\omega(j-1))\ti{a}(j)^2}{\sin^2(k_{j-1})
\ti{a}(j-1)^2}
\Bigg(2\cos(k_j)B_\omega(j)A_\omega(j)  \Bigg) \\
&\quad +\frac{\sin^2(\bar{\theta}_\omega(j-1))\ti{a}(j)^2}{\sin^2(k_{j-1})
\ti{a}(j-1)^2}
\left( \mean{A_\omega(j)^2} +\mean{B_\omega(j)^2} \right).
\end{split}
\eeq
Then the first three terms have mean zero and so, by Lemma 8.4 of \cite{efgp} we get that
\beq \no
\xi_{\calW} \equiv \lim_{n \rightarrow \infty} \frac{1}{F_\gamma(n)}\sum_{j=1}^\infty
\frac{\sin^2(\bar{\theta}_\omega(j-1))\ti{a}(j)^2}{\sin^2(k_{j-1})
\ti{a}(j-1)^2}
\left( \mean{A_\omega(j)^2} +\mean{B_\omega(j)^2} \right).
\eeq
Expanding $\mean{A_\omega(j)^2}$, throwing out terms that are $o(j^{-2\gamma})$, and applying Proposition \ref{theta-sums}
(writing, again, $\sin^2(\alpha)=\frac{1}{2}-\frac{\cos (2\alpha)}{2}$), we get
\beq \label{xiW1}
\begin{split}
\gamma_{\calW} &=\lim_{n \rightarrow \infty}\frac{1}{F_\gamma(n)} \sum_{j=1}^n
\frac{j^{2\eta_1}}{(j-1)^{2\eta_1}\sin^2(k_{j-1})}\Lambda j^{-2\gamma} \\
& \quad \times
\left(\mean{4Y_\omega(j)^2}+
\mean{X_\omega(j)^2}\right) \\
&=\lim_{n \rightarrow \infty}\frac{1}{F_\gamma(n)} \sum_{j=1}^n
2\Lambda j^{-2\gamma}= \left \{
\begin{array}{cc}
0 & {\rm if} \ \gamma>1/2 \\
2\Lambda & {\rm otherwise,}
\end{array} \right.
\end{split}
\eeq
(recall $\mean{Y_\omega(j)^2}=\frac{1}{4}$).
Applying the same procedure to $\xi_{\calZ \calW}$ and $\xi_{A}$ we get
\beq \label{xiZW}
\begin{split}
\xi_{\calZ \calW}&=\lim_{n \rightarrow \infty} \frac{2}{F_\gamma(n)} \sum_{j=1}^n
\frac{j^{2\eta_1}}{(j-1)^{2\eta_1}}\frac{\sin^2(\bar{\theta}_\omega(j-1))}{\sin^2(k_{j-1})}
\mean{2\Lambda\frac{Y_\omega(j)^2}{j^{2\gamma}}} \\
&=\lim_{n \rightarrow \infty} \frac{1}{F_\gamma(n)} \sum_{j=1}^n\frac{\Lambda}{2}
\frac{1}{j^{2\gamma}}=\left \{ \begin{array}{cc}
0 & {\rm if} \ \gamma>1/2 \\
\frac{\Lambda}{2} & {\rm otherwise}
\end{array} \right.
\end{split}
\eeq
and
\beq \label{xiA}
\begin{split}
\xi_{A}&=\lim_{n \rightarrow \infty} \frac{-1}{F_\gamma(n)}\sum_{j=1}^n
\mean{2 \Lambda \frac{Y_\omega(j)^2}{j^{2\gamma}}} \\
&=\lim_{n \rightarrow \infty} \frac{1}{F_\gamma(n)}\sum_{j=1}^n
\left(\frac{\Lambda}{2} \frac{1}{j^{2\gamma}}\right)
=\left \{ \begin{array}{cc}
0 & {\rm if} \ \gamma>1/2 \\
\frac{\Lambda}{2} & {\rm otherwise.}
\end{array} \right.
\end{split}
\eeq
The computation of $\xi_{\calZ \calW_2}$ involves $\sin^4(\bar{\theta}_\omega)=\frac{3}{8}-
\frac{\cos(2\bar{\theta}_\omega)}{2}+\frac{\cos(4 \bar{\theta}_\omega)}{8}$, for which Proposition
\ref{theta-sums} is useless. Luckily, the $\cos(4\bar{\theta}_\omega(j))$ cancels out. As before
\beq \no
\begin{split}
\xi_{\calZ \calW_2}&=\lim_{n \rightarrow \infty}  \frac{-2}{F_\gamma(n)}\sum_{j=1}^n
\frac{2\Lambda}{j^{2\gamma}} \\
&\quad \times \Bigg(\frac{\mean{4Y_\omega(j)^2} \sin^4(\bar{\theta}_\omega(j-1))}{\sin^2(k_{j-1})}
\frac{j^{2\eta_1}}{(j-1)^{2\eta_1}} \\
&\quad+\frac{\mean{X_\omega(j)^2}\sin^2(2\bar{\theta}_\omega(j-1))}{4}
\Bigg)
\end{split}
\eeq
Now write
\beq \no
\begin{split}
&\Bigg(\frac{\mean{4Y_\omega(j)^2} \sin^4(\bar{\theta}_\omega(j-1))}{\sin^2(k_{j-1})}
\frac{j^{2\eta_1}}{(j-1)^{2\eta_1}}
+\frac{\mean{X_\omega(j)^2}\sin^2(2\bar{\theta}_\omega(j-1))}{4}
\Bigg) \\
&\quad =\Bigg(\frac{\sin^4(\bar{\theta}_\omega(j-1))}{\sin^2(k_{j-1})}
\frac{j^{2\eta_1}}{(j-1)^{2\eta_1}}
+\frac{\sin^2(2\bar{\theta}_\omega(j-1))}{4}
\Bigg) \\
&\quad=\Bigg(\frac{3j^{2\eta_1}}{8(j-1)^{2\eta_1}\sin^2(k_{j-1})}+\frac{1}{8}\Bigg) \\
& \qquad +\frac{j^{2\eta_1}}{2(j-1)^{2\eta_1}\sin^2(k_{j-1})}\cos(2 \bar{\theta}_\omega(j-1)) \\
&\qquad+\cos(4 \bar{\theta}_\omega(j-1))\Bigg(\frac{j^{2\eta_1}}{(j-1)^{2\eta_1}\sin^2(k_{j-1})}-1 \Bigg)\\
&\quad=\Bigg(\frac{3j^{2\eta_1}}{8(j-1)^{2\eta_1}\sin^2(k_{j-1})}+\frac{1}{8}\Bigg) \\
&\qquad+\frac{j^{2\eta_1}}{2(j-1)^{2\eta_1}\sin^2(k_{j-1})}\cos(2 \bar{\theta}_\omega(j-1)) \\
&\qquad+\cos(4 \bar{\theta}_\omega(j-1))\Bigg(\mathcal{O}(n^{-1})+\mathcal{O}(n^{-2\eta_1})\Bigg).\\
\end{split}
\eeq
to see that
\beq \label{xiZW2}
\xi_{\calZ \calW_2}=\left \{ \begin{array}{cc}
0 & {\rm if} \ \gamma> 1/2 \\
-2\Lambda & {\rm otherwise,}
\end{array} \right.
\eeq
where the $\cos(2 \bar{\theta}_\omega)$ term vanishes by Proposition \ref{theta-sums}.
Summing up the various limits, the proposition is proved.
\end{proof}

\begin{proof}[Proof of Theorem \ref{transfer-asymp}]
By \eqref{compare-Rn-transfer}, the theorem follows from Proposition \ref{transfer-asymp1}.
\end{proof}

Theorem \ref{growing-weights} almost follows immediately from Theorem \ref{transfer-asymp} and Proposition
\ref{transfer-asymp1}. As in the Schr\"odinger situation, the case $\gamma=\frac{1}{2}$ requires some subtle reasoning.
We have established that, in this case, with probability one, equation \eqref{ev} has a solution, $\psi$, with
\beq \label{grow-asymp}
|\psi(n)|^2 \asymp n^{\Lambda-\eta_1}.
\eeq
In order to use subordinacy theory, we need the existence of another solution with faster
decay at infinity.
The following is Lemma 8.7 of \cite{efgp}, formulated for general regular matrices:

\begin{lemma} \label{osceledec-style}
Let $u_\phi=(\cos \phi, \sin \phi) \in \bbR^2$. For any matrix, $A \in GL_2(\bbR)$ with $\det(A)=d>0$, let $\phi(A)$ be the unique
$\phi \in (-\frac{\pi}{2}, \frac{\pi}{2}]$ with $\frac{1}{\sqrt{d}} \parallel Au_\phi \parallel = \sqrt{d} \parallel A
\parallel^{-1}$. Define $\rho(A)=\frac{\parallel Au_0 \parallel }{\parallel Au_{\pi/2}\parallel}$.

Let $A_n$ be a sequence of matrices in $GL_2(\bbR)$ with $\det(A_n)=d_n>0$, that satisfy
$\frac{1}{\sqrt{d_n}}\parallel A_n \parallel \rightarrow \infty$ and
$d_n \frac{\parallel A_{n+1} A_n^{-1} \parallel}{\parallel A_n \parallel \parallel A_{n+1} \parallel} \rightarrow 0$ as
$n \rightarrow \infty$. Let $\rho_n=\rho(A_n)$ and $\phi_n=\phi(A_n)$.
Then:
\begin{enumerate}
\item $\phi_n$ has a limit $\phi_\infty$ if and only if $\lim_{n \rightarrow \infty}\rho_n=\rho_\infty$ exists ($\rho_\infty=\infty$
is allowed, but then we only have $\vert \phi_n \vert \rightarrow \frac{\pi}{2})$.
\item Suppose $\phi_n$ has a limit $\phi_\infty \neq 0,\frac{\pi}{2}$ (equivalently, $\rho_\infty \neq 0, \infty$). Then
\beq \no
\lim_{n \rightarrow \infty}\frac{\log \parallel A_n u_\infty \parallel-\log \sqrt{d_n}}{\log \parallel A_n \parallel-\log
\sqrt{d_n}}=-1
\eeq
if and only if
\beq \no
\limsup_{n \rightarrow \infty}\frac{\log \vert \rho_n-\rho_\infty \vert}{\log \parallel A_n \parallel -\log \sqrt{d_n}}\leq -2.
\eeq
\end{enumerate}
\end{lemma}

\begin{proposition} \label{decaying-sol}
Let $J_{\Upsilon,\omega}$ be the family of random Jacobi matrices described
in Theorem \ref{transfer-asymp}, with $\gamma=\frac{1}{2}$.
Then, for any $E \in \bbR$, there exists, with probability one, an initial condition
$\Psi_{\phi(\omega)}=\left( \begin{array}{c} \cos(\phi(\omega)) \\ \sin(\phi(\omega)) \end{array} \right)$ such that
\beq \label{decay-asymp}
\lim_{n \rightarrow \infty}\frac{\log \parallel T_\omega^E(n)\Psi_{\phi(\omega)}\parallel}{\log n}
=-\frac{1}{2}\left(\Lambda+\eta_1 \right).
\eeq
\end{proposition}
\begin{proof}
We imitate the proof of \cite[Lemma 8.8]{efgp}. By Proposition \ref{transfer-asymp1}
\beq \no
\lim_{n \rightarrow \infty} \frac{\log \vert R_{\omega,0}(n) \vert}{\log n}=
\frac{1}{2}\left(\Lambda+\eta_1 \right)
\eeq
and
\beq \no
\lim_{n \rightarrow \infty} \frac{\log \vert R_{\omega, \pi/2}(n) \vert}{\log n}=
\frac{1}{2}\left(\Lambda+\eta_1 \right)
\eeq
for almost every $\omega$.
By \eqref{EFGP1}-\eqref{EFGP2},
\beq \no
R_{\omega,0}(n)R_{\omega,\pi/2}(n)\sin(\theta_{\omega,\pi/2}(n)-\theta_{\omega,0}(n))=\ti{a}(n) \sin k_n
\eeq
so that, for almost every $\omega$,
\beq \label{theta-asymp}
\lim_{n \rightarrow \infty} \frac{\log |\theta_{\omega,0}(n)-\theta_{\omega,\pi/2}(n)|}{\log n}=
-\Lambda
\eeq
(recall $\ti{a}(n)=\lambda_1 n^{\eta_1}$).

Let $\rho_\omega(n)=\frac{R_{\omega,0}(n)}{R_{\omega,\pi/2}(n)}$. Then
\beq \label{L-def}
L_\omega(n) \equiv \log \rho_\omega(n+1)-\log \rho_\omega(n)=\log(1+\mathcal{X}_{\omega,0}(n))-\log(1+\mathcal{X}_{\omega,\pi/2}(n))
\eeq
where
\beq \label{X-def}
\mathcal{X}_{\omega,\phi}(n)=\left( \parallel \calZ_{\omega,\phi}(n+1)+\calW_{\omega,\phi}(n+1) \parallel^2-1 \right).
\eeq
Since $\mathcal{X}_{\omega,\phi}(j)=\mathcal{O} \left(\frac{1}{\sqrt{j}} \right)$ almost surely, we may apply a finite Taylor
expansion to the above (using \eqref{W},\eqref{Z} and \eqref{ZW}) to see that, with probability one, for large enough
$n$
\beq \no
\begin{split}
L_\omega(n)&=
\Delta^1_\omega(n) \left(\sin^2(\bar{\theta}_{\omega,0}(n-1))-\sin^2(\bar{\theta}_{\omega,\pi/2}(n-1)) \right) \\
&\quad+\Delta^2_\omega(n) \left(\sin(2\bar{\theta}_{\omega,0}(n-1))-\sin(2\bar{\theta}_{\omega,\pi/2}(n-1)) \right) \\
&\quad+\Delta^3_\omega(n)
\Big( \sin(\bar{\theta}_{\omega,0}(n-1))\sin(\theta_{\omega,0}(n-1)) \\
&\quad-\sin(\bar{\theta}_{\omega,\pi/2}(n-1))\sin(\theta_{\omega,\pi/2}(n-1))  \Big) \\
&\quad +\mathcal{O}\left( n^{-\left(1+\Lambda-\varepsilon \right)}\right)
\end{split}
\eeq
where
\beq \no
\mean{\Delta^1_\omega(n)} = \mean{\Delta^2_\omega(n)} = \mean{\Delta^3_\omega(n)} = 0
\eeq
and
\beq \no
\mean { \left( \Delta^i_\omega(n) \right)^2} \leq \frac{C_j}{n}
\eeq
$i=1,2,3$. (The $\mathcal{O}\left( n^{-\left(1+\Lambda-\varepsilon \right)}\right)$ for the
remainder follows from Lemma \ref{almost-sure-decay} and \eqref{theta-asymp}.)

A standard application of Kolmogorov's inequality and the Borel Cantelli Lemma shows that with probability one, for large
enough $k$, and for each $i=1,2,3$
\beq \label{cond-delta}
\sup_{m=2^{k-1}+1,\ldots,2^{k}} \left \vert \sum_{j=2^{k-1}+1}^m \Delta^i_\omega(j) \right \vert \leq k
\eeq
and also
\beq \label{cond-delta1}
\sup_{m=2^{k-1}+1,\ldots,2^{k}-1} \left \vert \sum_{m}^{2^k} \Delta^i_\omega(j) \right \vert \leq k.
\eeq
Combining this with the fact that, with probability one, for large enough n,
\beq \no
\left(\sin^2(\bar{\theta}_{\omega,0}(n-1))-\sin^2(\bar{\theta}_{\omega,\pi/2}(n-1)) \right)
= \mathcal{O}\left(n^{-\Lambda+\varepsilon} \right),
\eeq
\beq \no
\left(\sin(2\bar{\theta}_{\omega,0}(n-1))-\sin(2\bar{\theta}_{\omega,\pi/2}(n-1)) \right)
=\mathcal{O}\left(n^{-\Lambda+\varepsilon} \right)
\eeq
and
\beq \no
\begin{split}
&\Big( \sin(\bar{\theta}_{\omega,0}(n-1))\sin(\theta_{\omega,0}(n-1)) \\
&\quad-\sin(\bar{\theta}_{\omega,\pi/2}(n-1))\sin(\theta_{\omega,\pi/2}(n-1))  \Big)
=
\mathcal{O}\left(n^{-\Lambda+\varepsilon} \right),
\end{split}
\eeq
it follows that
\beq \no
\sum_{n=1}^\infty L_\omega(n)
\eeq
exists and
\beq \no
\left \vert \sum_{n=N}^\infty L_\omega(n) \right \vert \leq C_\omega N^{-\Lambda+\varepsilon}
\eeq
almost surely.

Thus, with proability one,
$\lim_{n \rightarrow \infty}\frac{R_{\omega,0}(n)}{R_{\omega,\pi/2}(n)}=\lim_{n \rightarrow \infty}
\rho_\omega(n)=\rho_\omega(\infty)$ exists and
\beq \no
\limsup_{n \rightarrow \infty}\frac{\log \vert\rho_\omega(n)-\rho_\omega(\infty) \vert}{\log n} \leq -
\Lambda.
\eeq
Lemma \ref{osceledec-style} completes the proof (note that $d_n \equiv \det T_\omega^E(n)=\frac{1}{a_\omega(n)}$).
\end{proof}

We are now ready to complete the
\begin{proof}[Proof of Theorem \ref{growing-weights}]
\begin{enumerate}
\item By Theorem \ref{transfer-asymp}, Fubini's Theorem, the fact that the distribution of $X_\omega(n)$ is absolutely continuous with
respect to Lebesgue measure, and the theory of rank-one perturbations (\cite{rankone}), it follows that, with probability
one, the spectral measure is supported on the set of energies where, for any $\varepsilon>0$ and sufficiently large $n$,
$\parallel T^E_\omega(n) \parallel^2 \leq n^{-\eta_1+\varepsilon}$. From Corollary 4.4 of \cite{jit-last} it follows now
that the spectral measure is continuous with respect to $(1-\varepsilon)$-dimensional Hausdorff measure, for any
$\varepsilon>0$. Thus the spectral measure is one-dimensional. Since $\psi(n)=\ip{\delta_1}{\left( J-z \right)^{-1} \delta_n}$ 
solves the eigenvalue equation for $z$ (away from $n=0$), Theorem \ref{combes-thomas} and Wronskian conservation imply that 
if $z \in \bbR$ were to be outside of the spectrum, the transfer matrices would have to exhibit exponential growth. 
Since this is not the case, it follows that the spectrum is $\bbR$.
\item In this case again, the fact that the spectrum is $\bbR$ follows from the polynomial bound on the transfer matrices
in Theorem \ref{transfer-asymp} and Theorem \ref{combes-thomas} below.
As for the properties of the spectral measure, these follow from Theorem 1.2 in \cite{jit-last} using
\eqref{grow-asymp}, Proposition \ref{decaying-sol} and the theory of rank one perturbations.
\item The existence, with probability one, of an exponentially decaying eigenfunction, for every $E \in \bbR$, follows
from Theorem \ref{transfer-asymp} and Theorem 8.3 of \cite{last-simon}. Fubini and the theory of rank one perturbations
imply that the spectral measure is supported, with probability one, on the set where these eigenfunctions exist.
Comparing powers of $n$ in the exponent, 
Theorem \ref{combes-thomas} implies that, as long as $\eta_1>2\eta_2$, the spectrum fills  $\bbR$.
\end{enumerate}
\end{proof}


\appendix
\section{A Combes-Thomas Estimate for Jacobi Matrices with Unbounded Parameters}

This section presents a Combes-Thomas estimate, suitable for application to Jacobi matrices with unbounded
off-diagonal terms. Also see \cite{sahbani} for a related result.

\begin{theorem} \label{combes-thomas}
Let $J=\Jab$ be a self-adjoint Jacobi matrix such that $0<a(n) \leq f(n)$ for a nondecreasing function $f(n)$.
Let $z \in \bbC$ be such that ${\rm dist}(z, Spec(J))=\sigma$ (where $Spec(J)$ is the spectrum of $J$).
Then
\beq \label{combes-thomas1}
\left \vert \ip{\delta_1}{\left( J-z \right)^{-1} \delta_N} \right \vert
\leq \frac{2e}{\sigma} e^{-\alpha_N \cdot N}
\eeq
where $\alpha_N=\min \left(1, \frac{\sigma}{4e f(N)} \right)$.
In particular
\beq \label{combes-thomas2}
\left \vert \ip{\delta_1}{\left( J_{\Upsilon,\omega}-z \right)^{-1} \delta_N} \right \vert
\leq \frac{2e}{\sigma} e^{-C(\sigma, \omega) \cdot N^{1-\eta_1}}
\eeq
almost surely.
\end{theorem}
\begin{remark*}
The monotonicity of $f$ is not essential. One may instead replace $f(N)$ in the formula for $\alpha_N$, by
$\max(f(1),\ldots,f(N))$.
\end{remark*}
\begin{proof}
Let $R_N$ be the diagonal matrix defined by
\beq \label{definitionRN}
R_N(n,n)= \left \{ \begin{array}{cc}
e^{\alpha_N \cdot n} & n \leq N \\
e^{\alpha_N \cdot N} & n > N.
\end{array} \right.
\eeq
Then
\beq \no
e^{\alpha_n \cdot (N-1)} \ip{\delta_1}{\left( J-z \right)^{-1} \delta_N}=
\ip{\delta_1}{R_N^{-1} \left(J-z \right)^{-1} R_N \delta_N}
\eeq
so it suffices to bound $\parallel R_N^{-1} \left( J-z \right)^{-1} R_N \parallel$. Noting that
$R_N^{-1} \left( J-z \right)^{-1} R_N= \left(R_N^{-1} \left( J-z \right)R_N \right)^{-1} \equiv C_N(z)$, we may apply the Resolvent
identity to get
\beq \no
C_N(z)=\left(J-z \right)^{-1}+C_N(z)
\cdot \left( J-z-R_N^{-1} \left( J-z \right) R_N \right) \left(J-z \right)^{-1}.
\eeq
A simple computation shows that $\parallel J-z-R_N^{-1} \left( J-z \right) R_N \parallel \leq 2 e f(N) \alpha_N$
if $\alpha_N \leq 1$. Thus, by $\parallel \left( J-z \right)^{-1} \parallel \leq \frac{1}{\sigma}$, we see that
\beq \no
\parallel C_N(z) \parallel \leq \frac{1}{\sigma}+C_N(z) \frac{2 e f(N)}{\sigma} \alpha_N
\eeq
so, by $\alpha_N \leq \frac{1}{2} \cdot \frac{\sigma}{2 e f(N)}$, we see that
\beq \no
\parallel C_N(z) \parallel \leq \frac{2}{\sigma}
\eeq
which finishes the proof.
\end{proof}

{\bf Acknowledgments}
We are grateful to Peter Forrester and Uzy Smilansky for presenting us with the problem that led to this paper.
We also thank Yoram Last and Uzy Smilansky for many useful discussions. We thank the referee for useful remarks.

This research was supported in part by THE ISRAEL SCIENCE FOUNDATION (grant no.\ 1169/06) and by grant
no.\ \mbox{2002068} from the
United States-Israel Binational Science Foundation (BSF), Jerusalem, Israel.


\end{document}